\documentclass[11pt,reqno]{amsart}
\usepackage{amsmath,amssymb,amsthm}
\usepackage[numbers]{natbib}
\usepackage{fullpage}
\newtheorem{lemma}{Lemma}
\newtheorem{theorem}{Theorem}
\newcommand{\R}{\mathbb{R}}

\DeclareMathOperator{\relint}{relint}

\newcommand{\Ep}{{\mathrm{E}}}

\renewcommand{\Pr}{{\mathrm{P}}}
\DeclareMathOperator{\dist}{dist}

\begin{document}
\title{Detailed proof of Nazarov's inequality}
\author[Chernozhukov]{Victor Chernozhukov}
\author[Chetverikov]{Denis Chetverikov}
\author[Kato]{Kengo Kato}

\address[V. Chernozhukov]{
Department of Economics and Operations Research Center, MIT, 50 Memorial Drive, Cambridge, MA 02142, USA.}
\email{vchern@mit.edu}

\address[D. Chetverikov]{
Department of Economics, UCLA, Bunche Hall, 8283, 315 Portola Plaza, Los Angeles, CA 90095, USA.}
\email{chetverikov@econ.ucla.edu}

\address[K. Kato]{
Graduate School of Economics, University of Tokyo, 7-3-1 Hongo Bunkyo-ku, Tokyo 113-0033, Japan.}
\email{kkato@e.u-tokyo.ac.jp}

\date{\today.}

\thanks{We would like to thank Cl\'{e}ment Cerovecki for a comment.}

\begin{abstract}
The purpose of this note is to provide a detailed proof of Nazarov's inequality stated in Lemma A.1 in \cite{ChChKa17}. 
\end{abstract}


\maketitle

The statement of Lemma A.1 in \cite{ChChKa17}, which we called Nazarov's inequality, is as follows (we shall use here the different notation and clarify the constant). Unless otherwise stated, we shall follow the notational convention used in \cite{ChChKa17}. 

\begin{theorem}[Nazarov's inequality \cite{Nazarov03}; Lemma A.1 in \cite{ChChKa17}]
\label{lem: AC}
Let $Y=(Y_1,\dots,Y_p)'$ be a centered Gaussian random vector in $\R^p$ such that $\Ep[Y_j^2]\geq \underline{\sigma}^{2}$ for all $j=1,\dots,p$ and some constant $\underline{\sigma}>0$. 
Then for every $y \in \R^{p}$ and $\delta > 0$,
\[
\Pr(Y \leq y+\delta)-\Pr(Y \leq y)\leq \frac{\delta}{\underline{\sigma}} (\sqrt{2\log p}+2).
\]
\end{theorem}

This theorem can be alternatively stated as follows. 

\begin{theorem}
\label{lem: AC2}
Let $X \sim N(0,I_{p})$, and let $a_{1},\dots,a_{p} \in \mathbb{S}^{p-1} =\{ x \in \R^{p} : \| x \| = 1 \}$ and $b_{1},\dots,b_{p} \in \R$ be given. 
Then for every $\delta > 0$,
\[
\Pr \left (a_{j}'X \le b_{j} + \delta, \ 1 \le \forall j \le p \right) - \Pr \left (a_{j}'X \le b_{j}, \ 1 \le \forall j \le p \right) \le \delta (\sqrt{2\log p}+2).
\]
\end{theorem}

Theorem \ref{lem: AC} follows directly from Theorem \ref{lem: AC2}. Indeed, let $\Sigma = \Ep[YY']$, and observe that $Y \stackrel{d}{=} \Sigma^{1/2}X$. Let $\Sigma^{1/2} = (\sigma_{1},\dots,\sigma_{p})'$, and observe that 
\[
\Pr(Y \leq y+\delta)= \Pr \left \{ (\sigma_{j}/\| \sigma_{j} \|)' X \le y/\| \sigma_{j} \| + \delta/\| \sigma_{j} \|, \ 1 \le \forall j \le p \right ).
\]
Since $\| \sigma_{j} \|^{2} =\Ep[Y_{j}^{2}] \ge \underline{\sigma}^{2}$ for all $j$, the conclusion of Theorem \ref{lem: AC} follows from Theorem \ref{lem: AC2}. 

In what follows, we shall prove Theorem \ref{lem: AC2}.  For the notational convenience, let $K(t)=\{ x \in \R^{p} : a_{j}'x \le b_{j} + t, \ 1 \le \forall j \le p \}$ for $t \in \R$ with $K = K(0)$; we have to show that
\[
\Pr \left\{ X \in K(\delta) \setminus  K \right \} \le \delta (\sqrt{2\log p}+2)
\]
for every $\delta > 0$. 
To this end, define the function
\[
G (t) = \Pr \{ X \in K(t) \} =  \Pr \left \{ \max_{1 \le j \le p} (a_{j}'X -b_{j}) \le t \right \}, \ t \in \R. 
\]
Since $a_{j}'X-b_{j}, j=1,\dots,p$ are non-degenerate Gaussian, for any set $B \subset  \R$ with Lebesgue measure zero, 
\[
\Pr \left \{ \max_{1 \le j \le p} (a_{j}'X -b_{j})  \in B \right \} \le \sum_{j=1}^{p} \Pr \{  (a_{j}'X -b_{j})  \in B  \} = 0,
\]
so that $G$ is an absolutely continuous distribution function and hence is almost everywhere differentiable with 
\[
G(\delta) - G(0) = \int_{0}^{\delta} G'(t) dt = \int_{0}^{\delta} G_{+}'(t) dt,
\]
where $G_{+}'$ denotes the right derivative of $G$. We will show that $G$ is everywhere right differentiable with $G_{+}'(t) \le \sqrt{2\log p}+2$
for every $t \in \R$, which leads to the conclusion of Theorem \ref{lem: AC2}.

We begin with noting that by replacing $b_{j}$ with $b_{j}+t$, it suffices to show that $G$ is right differentiable at $t=0$ with $G_{+}'(0) \le \sqrt{2\log p}+2$. 
Therefore, the proof of Theorem \ref{lem: AC2} reduces to the following lemma.

\begin{lemma}
\label{lem: Klivans}
Work with the notation as stated above. Then  
\[
\lim_{\delta \downarrow 0} \frac{\Pr\{ X \in K(\delta) \setminus K\}}{\delta} \le \sqrt{2\log p}+2.
\]
\end{lemma}

Note that the existence of the limit on the left hand side is a part of the lemma. 

\begin{proof}[Proof of Lemma \ref{lem: Klivans}]
Let $\gamma_{p} = N(0,I_{p})$ and denote by $\varphi_{p}(x)$ the density of $\gamma_{p}$: 
\[
\varphi_{p}(x)=(2\pi)^{-p/2} e^{-\| x \|^{2}/2}, \ x \in \R^{p}.
\]
The set $K= \bigcap_{j=1}^{p} \{ x : a_{j}'x \le b_{j} \}$ is a convex polyhedron.
For $x \in \R^{p}$, let $P_{K}x$ denote the projection of $x$ to $K$, i.e., $\| x-P_{K}x \| = \min_{y \in K} \| x-y \|$. 
For a face $F$ of $K$, define 
\[
N_{F} = \{ x \in \R^{p} \setminus K : P_{K} x \in \relint (F) \} \quad \text{and} \quad N_{F}(\delta) = N_{F} \cap K(\delta),
\]
where $\relint (F)$ denotes the relative interior of $F$ (by ``face'', we mean a proper face).
Then $K(\delta) \setminus K = \bigcup_{F: \ \text{face of} \ K} N_{F}(\delta)$. Let $\dist (x,F) = \inf \{ \| x-y \| : y \in F \}$ denote the Euclidean distance between $x$ and $F$. Now, for any face $F$ of $K$ with dimension at most $p-2$, 
since $N_{F}(\delta) \subset F^{C\delta}= \{ x \in \R^{p} : \dist (x,F) \le C\delta \}$ for some sufficiently large constant $C$ (that may depend on $p$), we have $\gamma_{p}(N_{F}(\delta)) = O(\delta^{2}) = o(\delta)$ as $\delta \downarrow 0$. 
Hence, 
\[
\gamma_{p} (K(\delta) \setminus K ) = \gamma_{p} \left ( \bigcup_{F: \ \text{facet of} \ K} N_{F}(\delta) \right ) + o(\delta)
\]
as $\delta \downarrow 0$.
In addition, for two distinct facets $F_{1}$ and $F_{2}$ of $K$, the sets $N_{F_{1}}(\delta)$ and $N_{F_{2}}(\delta)$ are disjoint, and so the problem reduces to proving that
\[
\sum_{F: \ \text{facet of} \ K} \lim_{\delta \downarrow 0} \frac{\gamma_{p} (N_{F}(\delta))}{\delta} \le \sqrt{2\log p} + 2.
\]
Recall that a facet of $K$ is a face of $K$ with dimension $p-1$. We shall prove the following lemma.

\begin{lemma}
\label{lem: Nazarov}
For any facet $F$ of $K$, 
\begin{equation}
\label{eq: Nazarov}
\lim_{\delta \downarrow 0} \frac{\gamma_{p}(N_{F}(\delta))}{\delta} = \int_{F} \varphi_{p}(x) d \sigma (x) \le (\dist (0,F) + 1) \gamma_{p}(N_{F}),
\end{equation}
where $d \sigma(x)$ denotes the standard surface measure on $F$.
\end{lemma}

\begin{proof}[Proof of Lemma \ref{lem: Nazarov}]
The proof is essentially due to \cite{Nazarov03}, p.171-172.
Let $v$ be the outward unit normal vector to $\partial K$ at $F$. Then $N_{F} = \{ x+tv : x \in \relint(F), t > 0 \}$ and $N_{F} (\delta) = \{ x+tv: x \in \relint(F), 0 < t \le \delta \}$. 
Pick any  $\overline{x} \in F$. 
Then the set $[ F-\overline{x} ] = \{  x-\overline{x} : x \in F \}$ is contained in a linear subspace of $\R^{p}$ with dimension $p-1$.  
Let $\{ q_{1},\dots,q_{p-1} \}$ be an orthonormal basis of the linear subspace.
Then every $x \in F$ can be parameterized as 
\[
x=x (z_{1},\dots,z_{p-1})= \overline{x} + z_{1} q_{1} + \cdots + z_{p-1} q_{p-1}. 
\]
Hence every $y \in N_{F}$ can be parameterized as
\begin{align*}
y = y(z_{1},\dots,z_{p-1},t) = x(z_{1},\dots,z_{p-1})+tv = \overline{x} + z_{1} q_{1} + \cdots + z_{p-1} q_{p-1} + t v.
\end{align*}
Think of $(z_{1},\dots,z_{p-1},t)$ as the coordinates on $N_{F}$. Since the Jacobian matrix with respect to the change of variables $(z_{1},\dots,z_{p-1},t)' \to y$ is $[q_{1},\dots,q_{p-1},v]$ and the matrix $[q_{1},\dots,q_{p-1},v]$ is orthogonal, we have 
\begin{equation}
\label{eq: gaussian_volume}
\begin{split}
\gamma_{p}(N_{F}) &= \int_{y(z_{1},\dots,z_{p-1},t)  \in N_{F}} \varphi_{p} (y(z_{1},\dots,z_{p-1},t) )dz_{1}\cdots dz_{p-1} dt \\
&=\int_{x(z_{1},\dots,z_{p-1}) \in F} \left ( \int_{0}^{\infty} \varphi_{p} (y(z_{1},\dots,z_{p-1},t) ) dt \right ) dz_{1} \cdots dz_{p-1}.
\end{split}
\end{equation}
Similarly, we have
\[
\gamma_{p}(N_{F}(\delta)) = \int_{0}^{\delta} \left ( \int_{x(z_{1},\dots,z_{p-1}) \in F} \varphi_{p} (y(z_{1},\dots,z_{p-1},t) ) dz_{1} \cdots dz_{p-1} \right ) dt,
\]
and hence
\[
\lim_{\delta \downarrow 0} \frac{\gamma_{p}(N_{F}(\delta))}{\delta} = \int_{x(z_{1},\dots,z_{p-1}) \in F} \varphi_{p} (x(z_{1},\dots,z_{p-1}) ) dz_{1} \cdots dz_{p-1} = \int_{F} \varphi_{p}(x) d\sigma(x).
\]
So, we have proved the first equation in (\ref{eq: Nazarov}). 

To prove the second inequality in (\ref{eq: Nazarov}),  observe that for any $x \in F$, 
\begin{equation}
\label{eq: intermediate_inequality}
 \int_{0}^{\infty} \varphi_{p} (x + tv) dt \ge \varphi_{p} (x) \int_{0}^{\infty} e^{-t |x'v|} e^{-t^{2}/2} dt \ge \frac{\varphi_{p}(x)}{|x'v| + 1}.
 \end{equation}
To verify the last inequality, let $r =|x'v|$ and observe that $\int_{0}^{\infty} e^{-tr - t^{2}/2} dt = e^{r^{2}/2} \int_{r}^{\infty} e^{-t^{2}/2} dt$, and so it suffices to show that $\psi (r) := \int_{r}^{\infty} e^{-t^{2}/2} dt - e^{-r^{2}/2}/(r+1) \ge 0$ for all $r \ge 0$. But this follows from the observation that $\psi'(r) = -re^{-r^{2}/2}/(r+1)^{2} \le 0$ for all $r \ge 0$ and $\lim_{r \to \infty} \psi (r) = 0$. Combining (\ref{eq: gaussian_volume}) and (\ref{eq: intermediate_inequality}), we have 
\begin{align*}
\gamma_{p}(N_{F}) &\geq \int_{x(z_{1},\dots,z_{p-1}) \in F} \left ( \frac{1}{|x(z_{1},\dots,z_{p-1})'v| + 1} \right ) \varphi_{p}(x(z_{1},\dots,z_{p-1}) ) dz_{1}\cdots dz_{p-1} \\
 &= \int_{F} \left ( \frac{1}{|x'v|+ 1} \right ) \varphi_{p}(x) d \sigma (x). 
\end{align*}
The value of $x'v$ is invariant for any $x$ in the hyperplane containing $F$, and so choosing $x = \dist (0,F) v$ or $x=-\dist (0,F) v$, we have $|x'v| = \dist (0,F)$. Therefore, we conclude that 
\[
\gamma_{p}(N_{F}) \ge \frac{1}{\dist (0,F) + 1} \int_{F} \varphi_{p}(x) d\sigma(x).
\]
This completes the proof of Lemma \ref{lem: Nazarov}. 
\end{proof}
Return to the proof of  Lemma \ref{lem: Klivans}. The rest of the proof is essentially due to \cite{Klivans08}. So far, we have shown that 
\[
\lim_{\delta \downarrow  0} \frac{\gamma_{p}(K(\delta) \setminus K)}{\delta} = \sum_{F: \ \text{facet of} \ K} \int_{F} \varphi_{p}(x) d\sigma(x),
\]
and for each facet $F$ of $K$, $\int_{F} \varphi_{p}(x) d\sigma(x) \le (\dist (0,F) + 1) \gamma_{p}(N_{F})$. 
Decompose the sum $\sum_{F: \ \text{facet of} \ K}$ into two parts:
\[
\sum_{F: \ \text{facet of} \ K} \int_{F} \varphi_{p}(x) d\sigma(x) =\sum_{\dist (0,F) > \sqrt{2 \log p}} + \sum_{\dist (0,F) \leq \sqrt{2 \log p}}.
\]
First, consider the case where $\dist (0,F) > \sqrt{2\log p}$, and suppose that $\overline{x} = \dist(0,F) v$ is in the hyperplane containing $F$ (otherwise take $\overline{x} =- \dist (0,F)v$), where $v$ denotes the outward unit normal vector to $\partial K$ at $F$. Using the  coordinates appearing in the proof of Lemma \ref{lem: Nazarov}, we can parameterize $x \in F$  as 
\[
x = x(z_{1},\dots,z_{p-1}) = \overline{x} + z_{1}q_{1}+\cdots+z_{p-1}q_{p-1}.
\]
Letting $z= (z_{1},\dots,z_{p-1})'$, we have 
\[
\varphi_{p}(x) = (2\pi)^{-p/2} e^{-\| \overline{x} \|^{2}/2} e^{-\| x-\bar{x} \|^{2}/2} = \varphi_{1}(\| \overline{x} \|) \varphi_{p-1}(z) \leq (2\pi)^{-1/2} p^{-1} \varphi_{p-1}(z),
\]
where we have used that $\| \overline{x} \| = \dist (0,F) > \sqrt{2\log p}$. Hence, using the fact that the number of facets of $K$ is at most $p$, we have 
\[
\begin{split}
&\sum_{\dist (0,F) > \sqrt{2\log p}} \int_{F} \varphi_{p}(x) d\sigma(x) \le (2\pi)^{-1/2} p^{-1} \sum_{F: \ \text{facet of} \ K} \underbrace{\int_{x(z_{1},\dots,z_{p-1}) \in F} \varphi_{p-1}(z) dz}_{\le 1} \\
&\qquad \le (2\pi)^{-1/2} p^{-1} \times p = (2\pi)^{-1/2} \le 1.
\end{split}
\]
Next, if $\dist (0,F) \le \sqrt{2\log p}$, then 
\[
\int_{F} \varphi_{p}(x) d\sigma(x) \le (\sqrt{2\log p}+1) \gamma_{p}(N_{F}).
\]
Hence, 
\[
\sum_{\dist (0,F) \le \sqrt{2\log p}} \int_{F} \varphi_{p}(x) d\sigma(x) \le (\sqrt{2\log p}+1) \underbrace{\sum_{\dist (0,F) \le \sqrt{2\log p}}  \gamma_{p}(N_{F})}_{\le \gamma_{p}(\R^{p})=1} \le \sqrt{2\log p}+1.
\]
This completes the proof. 
\end{proof}

\end{document}